\documentclass{amsart}
\usepackage[T1]{fontenc}
\usepackage[utf8]{inputenc} 
\usepackage[english]{babel}

\usepackage{amsmath, amssymb, amsthm, mathtools}

\usepackage{graphicx}
\usepackage{xcolor}
\usepackage{hyperref}

\hypersetup{
    colorlinks=true,
    linkcolor=blue,
    citecolor=blue,
    urlcolor=blue
}

\newtheorem{theorem}{Theorem}[section]
\newtheorem{proposition}[theorem]{Proposition}
\newtheorem{lemma}[theorem]{Lemma}
\newtheorem{corollary}[theorem]{Corollary}

\theoremstyle{definition}
\newtheorem{definition}[theorem]{Definition}

\newtheorem{remark}[theorem]{Remark}

\numberwithin{equation}{section}

\newcommand{\R}{\mathbb{R}}
\newcommand{\C}{\mathbb{C}}
\newcommand{\Q}{\mathbb{Q}}
\newcommand{\Z}{\mathbb{Z}}

\begin{document}

\title[Four squares from three real polynomials]{Four squares from three real polynomials}
\author[A. Jurasi\'c]{Ana Jurasi\'c}

\address[Ana Jurasi\'c]{Faculty of Mathematics\\University of Rijeka\\Radmile Matej\v{c}i\'c 2\\51000 Rijeka\\Croatia}
\email[A. Jurasi\'c]{ajurasic@math.uniri.hr}


\begin{abstract}
We construct three nonconstant polynomials $a,b,c\in \R[X]$ such that
\[
   ab+1,\qquad ac+1,\qquad bc+1,\qquad abc+1
\]
are all squares in $\R[X]$. The entries in the construction are quadratic, so the construction has the smallest possible positive degrees. By composition with nonconstant polynomials, this gives infinitely many examples over $\R[X]$ with all entries nonconstant.
\end{abstract}

\maketitle

\section{Introduction}\label{sec:1}

A Diophantine $m$-tuple is classically a set of $m$ distinct positive integers for which the product of any two distinct elements, increased by $1$, is a square. Polynomial analogues and $D(n)$-variants have a substantial literature. Early polynomial versions were studied by Jones \cite{Jones1976}. Foundational work on polynomial variants includes \cite{DujellaFuchs2004,DFT,DFW,DJsize,DJquad,Jquad}, and a recent general reference is \cite{Dbook}. The present note concerns a related, but stronger, condition for triples: besides requiring the three pairwise products plus $1$ to be squares, we also require the product of all three entries plus $1$ to be a square.

\begin{definition}
A set $\{a,b,c\}\subset \R[X]$ of three nonzero polynomials is called a \emph{four-square polynomial triple} if there exist $r,s,t,u\in \R[X]$, such that
\begin{equation}\label{eq:four-square}
   ab+1=r^2,\qquad ac+1=s^2,\qquad bc+1=t^2,\qquad abc+1=u^2.
\end{equation}
\end{definition}

The corresponding problem for positive integers asks for $a,b,c>1$ satisfying the four square conditions in \eqref{eq:four-square}. Dujella and Szalay \cite{DS} proved that there are infinitely many such triples of positive integers. Their construction lies inside regular Diophantine triples.

If an entry equal to $1$ is allowed, the problem degenerates. For example, $\{1,\ k^2-1,\ (k+1)^2-1\}$ satisfies the four required square conditions in the integer setting, and the same identity works over $\Z[X]$. Therefore, as we will explain in the Section \ref{sec:2}, the natural polynomial question is whether all three entries can be distinct nonconstant polynomials. We answer this affirmatively over $\R[X]$, and in fact with quadratic entries.

\section{Elementary restrictions}\label{sec:2}

As in the classical case, it is natural to assume that at least one element of a four-square polynomial triple is nonconstant. Also, analogously as in classical case, we can't have two equal polynomials in such a set. In this section we record the other three simple restrictions. The first explains why constant entries, apart from $0$ and $1$, are not possible in a four-square polynomial triple.

\begin{lemma}\label{lem:constant-entry}
Let $a,b,c\in \R[X]$ satisfy \eqref{eq:four-square}. Suppose that one entry, say $a$, is a constant in $\R\setminus\{0,1\}$. Then, either both $b$ and $c$ are constant, or exactly one of $b$ and $c$ is equal to zero and the other one is a nonzero polynomial.
\end{lemma}

\begin{proof}
From $bc+1=t^2$ and $abc+1=u^2$, we obtain
\[
   u^2-a t^2=1-a.
\]
Since $a\ne 0,1$,
\[
   (u-t\sqrt{a})(u+t\sqrt{a})=1-a
\]
is a factorization of a nonzero constant in $\C[X]$. Thus both factors are units in $\C[X]$, hence constants. Therefore $u$ and $t$ are constant, so $bc=t^2-1$ is constant. This is possible only if $b$ and $c$ are both constant, or exactly one of $b$ and $c$ is equal to zero and the other one is a nonzero polynomial.
\end{proof}

The exceptional constants $0$ and $1$ do give degenerate families. For any $k\in \R[X]$,
\[
   \{0,\ k^2-1,\ (k+1)^2-1\}\ \ \textup{and}\ \ \{1,\ k^2-1,\ (k+1)^2-1\}
\]
satisfy \eqref{eq:four-square}, since
\[
   (k^2-1)((k+1)^2-1)+1=(k^2+k-1)^2.
\] By Lemma \ref{lem:constant-entry}, any entry from $\R\setminus\{0,1\}$ yields either a set $\{a,b,c\}$ containing $0$ or a set consisting entirely of constants.

\begin{lemma}\label{lem:degree-leading}
Let $\{a,b,c\}\subset \R[X]$ be a four-square polynomial triple whose three entries are nonconstant. Then $\deg a$, $\deg b$, and $\deg c$ are all even. Moreover, the leading coefficients of $a,b,c$ are positive real numbers.
\end{lemma}

\begin{proof}
Since $ab+1$, $ac+1$, and $bc+1$ are nonconstant squares,
\[
   \deg a+\deg b\equiv \deg a+\deg c\equiv \deg b+\deg c\equiv 0 \pmod 2.
\]
Thus, all three degrees have the same parity. Since $abc+1$ is also a nonconstant square, $\deg a+\deg b+\deg c$ is even. The common parity is therefore even.

Let $A,B,C$ be the leading coefficients of $a,b,c$, respectively. The leading coefficients of $ab+1$, $ac+1$, $bc+1$, and $abc+1$ must be positive. Hence, $AB>0$, $AC>0$, $BC>0$, and $ABC>0$. It follows that $A,B,C$ are all positive.
\end{proof}

We also note why the identity used in the integer construction does not immediately carry over to polynomials. Dujella and Szalay \cite{DS} use integer solutions of the equation
\begin{equation}\label{eq:pell}
   x^2-4xy+y^2=1.
\end{equation}
The direct polynomial analogue of this equation has no nonconstant solutions.

\begin{lemma}\label{lem:pell-constant}
If $x,y\in \R[X]$ satisfy \eqref{eq:pell}, then $x$ and $y$ are constant.
\end{lemma}

\begin{proof}
From \eqref{eq:pell}, we have
\[
   \bigl(x-(2+\sqrt3)y\bigr)\bigl(x-(2-\sqrt3)y\bigr)=1.
\]
Both factors are units in $\R[X]$, so they are constant. Solving the resulting two linear equations gives that $x$ and $y$ are constant.
\end{proof}

\section{A quadratic construction}\label{sec:3}

The construction begins with the following elementary observation:

\begin{lemma}\label{lem:basic-square}
Let $u,v, D\in \R$ with $u\ne v$, and let $D>0$. Then
\[
   1+D(Y-u)(Y-v)
\]
is a square of a linear polynomial in $Y$ if and only if
\[
   D=\frac{4}{(v-u)^2}.
\]
For this value of $D$,
\begin{equation}\label{eq:basic-square}
   1+\frac{4(Y-u)(Y-v)}{(v-u)^2}
   =\left(\frac{2Y-u-v}{v-u}\right)^2.
\end{equation}
\end{lemma}

\begin{proof}
The discriminant of polynomial
\[
   D(Y-u)(Y-v)+1=DY^2-D(u+v)Y+Duv+1
\]
is
\[
   D^2(u+v)^2-4D(Duv+1)=D\bigl(D(u-v)^2-4\bigr).
\]
Since $D>0$, this discriminant vanishes exactly when $D=4/(u-v)^2$. The identity \eqref{eq:basic-square} follows by substitution.
\end{proof}

We now give the construction. Let \(Y=X^2\) and look for polynomials from $\R[X]$ of the form
\begin{equation}\label{eq:abc-def}
   a=A(X^2-\alpha),\qquad
   b=B(X^2-\beta),\qquad
   c=C(X^2-\gamma),
\end{equation}where $A,B,C\neq 0$ and $\alpha,\beta,\gamma$ are distinct. The idea is to make such a construction for which the first three square conditions in \eqref{eq:four-square} are automaticaly fulfiled and the final
condition collapses to a square. Hence, we impose
\begin{equation}\label{eq:lambda-pair-products}
   AB=\frac{4}{(\beta-\alpha)^2},\qquad
   AC=\frac{4}{(\gamma-\alpha)^2},\qquad
   BC=\frac{4}{(\gamma-\beta)^2}.
\end{equation}A convenient compatible choice is
\begin{equation}\label{eq:lambda-defs}
   A=\frac{2(\gamma-\beta)}{(\beta-\alpha)(\gamma-\alpha)},\qquad
   B=\frac{2(\gamma-\alpha)}{(\beta-\alpha)(\gamma-\beta)},\qquad
   C=\frac{2(\beta-\alpha)}{(\gamma-\alpha)(\gamma-\beta)}.
\end{equation}Set
\begin{equation}\label{eq:delta}
   \Delta=(\beta-\alpha)(\gamma-\alpha)(\gamma-\beta).
\end{equation}Then, equations in \eqref{eq:lambda-defs} also give
\begin{equation}\label{eq:lambda-product}
   ABC=\frac{8}{\Delta}.
\end{equation}The only remaining condition
is the fourth square condition \(abc+1\in \R[X]^2\). 

By \eqref{eq:lambda-product}, 
\begin{equation}\label{eq:8} abc+1=\frac{8}{\Delta}(Y-\alpha)(Y-\beta)(Y-\gamma)+1. 
\end{equation}We now choose \(\alpha,\beta,\gamma\) so that
\begin{equation}\label{eq:cubic-target}
  (Y-\alpha)(Y-\beta)(Y-\gamma)=Y(Y-1)^2-d,
\end{equation}
for some nonzero constant \(d\). Indeed, after a substitution \(Y=X^2\),
\[
  Y(Y-1)^2=X^2(X^2-1)^2=\bigl(X(X^2-1)\bigr)^2
\]
is already a square in \(\R[X]\). Substituting \eqref{eq:cubic-target} into \eqref{eq:8}, and taking \begin{equation}\label{eq:Delta-equals-8d}
  \Delta=8d,
\end{equation} gives
\begin{equation}\label{eq:fourth}
  abc+1=\frac{8}{\Delta}\bigl(Y(Y-1)^2-d\bigr)+1=\frac{8}{\Delta}Y(Y-1)^2.
\end{equation} The problem is now reduced to finding a cubic polynomial
\begin{equation}\label{eq:f}
  f(Y)=Y^3-2Y^2+Y-d,
\end{equation}
whose roots \(\alpha,\beta,\gamma\) satisfy \eqref{eq:Delta-equals-8d}. The discriminant of the polynomial \eqref{eq:f} is
\begin{equation}\label{eq:disc-f}
  \operatorname{disc}(f)=d(4-27d).
\end{equation}On the other hand, since $f$ is monic, $\operatorname{disc}(f)=\Delta^2$. Hence, by \eqref{eq:Delta-equals-8d} and \eqref{eq:disc-f}, we get $d(4-27d)=64d^2.$ Since \(d=0\) gives repeated roots of $f$ and then zero
denominators in \eqref{eq:lambda-defs}, division by \(d\ne 0\) gives
\begin{equation}\label{eq:d-value}
  d=\frac{4}{91}.
\end{equation} In \eqref{eq:f}, now we have
\begin{equation}\label{eq:cubic}
   f(Y)=Y^3-2Y^2+Y-\frac{4}{91}.
\end{equation}
The discriminant of $f$ is
\[
   \operatorname{disc}(f)=\frac{1024}{8281}=\left(\frac{32}{91}\right)^2>0,
\]
so $f$ has three distinct real roots $\alpha$, $\beta$, $\gamma$. If we set
\[
   \alpha<\beta<\gamma,
\] then $\Delta>0$ and we have
\begin{equation}\label{eq:delta-value}
   \Delta=\frac{32}{91}.
\end{equation}

\begin{theorem}\label{thm:main}
The polynomials $a,b,c$ defined by \eqref{eq:abc-def} form a four-square polynomial triple in $\R[X]$. More precisely,
\begin{align}
   ab+1&=\left(\frac{2X^2-\alpha-\beta}{\beta-\alpha}\right)^2, \label{eq:ab-square}\\
   ac+1&=\left(\frac{2X^2-\alpha-\gamma}{\gamma-\alpha}\right)^2, \label{eq:ac-square}\\
   bc+1&=\left(\frac{2X^2-\beta-\gamma}{\gamma-\beta}\right)^2, \label{eq:bc-square}\\
   abc+1&=\left(\frac{\sqrt{91}}{2}X(X^2-1)\right)^2. \label{eq:abc-square}
\end{align}
In particular, $a,b,c$ are different nonconstant polynomials. Their degrees are minimal among all four-square polynomial triples over $\R[X]$ with nonconstant entries.
\end{theorem}

\begin{proof}
The numbers $A,B,C$ are positive because $\alpha<\beta<\gamma$, and therefore $a,b,c$ are nonconstant polynomials in $\R[X]$. They are pairwise distinct: if, for example, $A(X^2-\alpha)=B(X^2-\beta)$, then comparison of coefficients gives $A=B$ and $\alpha=\beta$, a contradiction. The identities \eqref{eq:ab-square}, \eqref{eq:ac-square}, and \eqref{eq:bc-square} follow from \eqref{eq:basic-square} and \eqref{eq:lambda-pair-products}, with $Y=X^2$. The equation \eqref{eq:abc-square} follows from \eqref{eq:fourth} and \eqref{eq:delta-value}. The last assertion follows from Lemma \ref{lem:degree-leading}: in any example with all entries nonconstant, all three degrees are positive even integers, hence at least $2$.
\end{proof}

\begin{remark}\label{rem:numerical}
For reference, the roots of the polynomial \eqref{eq:cubic} are 
\[
   \alpha\approx 0.0485571496,\qquad
   \beta\approx 0.7594142315,\qquad
   \gamma\approx 1.1920286189.
\]
The corresponding leading coefficients of the polynomials $a$, $b$ and $c$ are
\[
   A\approx 1.0644452464,\qquad
   B\approx 7.4365598181,\qquad
   C\approx 2.8739949355.
\]
The exact form of the triple $\{a,b,c\}$ is the root-based one in \eqref{eq:abc-def}.
\end{remark}

\section{Infinitely many examples}\label{sec:4}

Now, it is easy to generate infinitely many four-square polynomial triples.

\begin{proposition}\label{prop:composition}
Let $\{a,b,c\}\subset \R[X]$ be a four-square polynomial triple whose three entries are nonconstant. If $h\in \R[X]$ is nonconstant, then
\[
   \{a\circ h,\ b\circ h,\ c\circ h\}
\]
is again a four-square polynomial triple with three nonconstant entries.
\end{proposition}

\begin{proof}
Substituting $X\mapsto h(X)$ in \eqref{eq:four-square} gives
\[
   a(h)b(h)+1=r(h)^2,\qquad
   a(h)c(h)+1=s(h)^2,
\]
\[
   b(h)c(h)+1=t(h)^2,\qquad
   a(h)b(h)c(h)+1=u(h)^2.
\]
A composition of two nonconstant polynomials is nonconstant, so the three entries remain nonconstant. The substitution homomorphism $p(X)\mapsto p(h(X))$ is injective, because $h$ is nonconstant. Hence, distinct entries remain distinct.
\end{proof}

\begin{corollary}\label{cor:infinite}
There are infinitely many four-square polynomial triples in $\R[X]$ with all entries nonconstant.
\end{corollary}

\begin{proof}
Apply Proposition \ref{prop:composition} to the triple from Theorem \ref{thm:main}, with $h_m(X)=X^m$, $m=1,2,3,\ldots$. The degree of polynomials in the resulting triple is $2m$, so these triples are distinct for infinitely many $m$.
\end{proof}

\section{Final remarks}\label{sec:final-remarks}

We described a construction over the real polynomials. It does not by itself give a solution over $\Z[X]$ or $\Q[X]$, because the coefficients in \eqref{eq:abc-def} are defined from the real roots of the cubic polynomial \eqref{eq:cubic}. Moreover, no evident coefficient rescaling preserves equations of the form ``product plus $1$ is a square''. Determining whether an analogue with all entries nonconstant exists over $\Q[X]$ or $\Z[X]$ remains a separate arithmetic question.\\

\section{Acknowledgements}
This work was supported by the Croatian Science Foundation Grant No.~IP-2022-10-5008. A.~J. was also supported by the European Union -- NextGenerationEU, project number uniri-iz-25-62-ALGEBRA. During the exploratory phase of this research, the author utilized GPT-5.5 Pro to investigate examples and candidate constructions. All mathematical results were performed and validated by the author. The author takes full responsibility for the mathematical correctness of these results. The author would like to thank Professor Andrej Dujella for his support and suggestions.

\bigskip

\end{document}